\documentclass{article}
\usepackage[a4paper]{geometry}
\title{Greedy Randomized and Maximal Weighted Residual Kaczmarz Methods with Oblique Projection}
\author{Fang Wang, Weiguo Li$^*$, Wendi Bao, Li Liu\\
\small China University of Petroleum\\[-0.8ex]
\small Qingdao, China\\
\small\tt $^*$liwg@upc.edu.cn\\
}
\date{June 2021}
\usepackage{algorithm}
\usepackage{amsthm}
\newtheorem{theorem}{Theorem}
\newtheorem{lemma}{Lemma}
\newtheorem{example}{Example}
\usepackage{algorithmicx}
\usepackage{algpseudocode}
\usepackage{booktabs}
\usepackage{graphicx}
\usepackage{amsmath}
\usepackage{charter}
\usepackage[charter]{mathdesign}
\usepackage{subfigure}
\usepackage{multirow}
\usepackage{algorithm,algpseudocode,float}
\usepackage{lipsum}
\makeatletter
\newenvironment{breakablealgorithm}
  {
   \begin{center}
     \refstepcounter{algorithm}
     \hrule height.8pt depth0pt \kern2pt
     \renewcommand{\caption}[2][\relax]{
       {\raggedright\textbf{\ALG@name~\thealgorithm} ##2\par}%
       \ifx\relax##1\relax 
         \addcontentsline{loa}{algorithm}{\protect\numberline{\thealgorithm}##2}%
       \else 
         \addcontentsline{loa}{algorithm}{\protect\numberline{\thealgorithm}##1}%
       \fi
       \kern2pt\hrule\kern2pt
     }
  }{
     \kern2pt\hrule\relax
   \end{center}
  }
\makeatother
\begin{document}
\bibliographystyle{plain}
\maketitle
\begin{abstract}
For solving large-scale consistent linear system, we combine two efficient row index selection strategies with Kaczmarz-type method with oblique projection, and propose a greedy randomized Kaczmarz method with oblique projection (GRKO) and the maximal weighted residual Kaczmarz method with oblique projection (MWRKO) .  Through those method, the number of iteration steps and running time can be reduced to a greater extent to find the least-norm solution, especially when the rows of matrix A are close to linear correlation. Theoretical proof and numerical results show that GRKO method and MWRKO method are more effective than greedy randomized Kaczmarz method and maximal weighted residual Kaczmarz method respectively.

\end{abstract}
\section{Introduction}\label{sec1}
Consider to solve a large-scale consistent linear system
\begin{align}
Ax=b,\label{Ax=b}
\end{align}
where the matrix $A\in R^{m\times n}$, $b\in R^m$. One of the solutions of the system (\ref{Ax=b}) is $x^*=A^\dag b$, which is the least Euclidean norm solution. Especially, when the coefficient matrix A is full column rank, $x^*$ is the unique solution of the system (\ref{Ax=b}).

There are many researches on solving the system (\ref{Ax=b}) through iterative methods, among which the Kaczmarz method is a representative and efficient row-action method. The Kaczmarz method \cite{K37} selects the rows of the matrix $A$ by using the cyclic rule, and in each iteration, the current iteration point is orthogonally projected onto the corresponding hyperplane. Due to its simplicity and performance, the Kaczmarz method has been applied to many fields, such as computerized tomography \cite{WH09,SH14}, image reconstruction \cite{DJ18,TD14,RM12,GV14}, distributed computing \cite{AD17}, and signal processing \cite{VN16,WH09,SH14}; and so on \cite{JJ17,HR94,XV02,XJ93}. Since the Kaczmarz method cycles through the rows of $A$, the performance may depend heavily on the ordering of these rows. A poor ordering may result in a very slow convergence rate.  McCormick \cite{S77} proposed a maximal weighted residual Kaczmarz (MWRK) method, and proved its convergence. In recent work, a new theoretical convergence estimate was proposed for the MWRK method in \cite{DH19}. Strohmer and Vershynin \cite{SV09} proposed a randomized Kaczmarz (RK) method which selects a given row with proportional to the Euclidean norm of the rows of the coefficient matrix $A$, and proved its convergence. After the above work, research on the Kaczmarz-type methods was reignited recently, see for example, the randomized block Kaczmarz-type methods \cite{DJ14,DR15,I19}, the greedy version of Kaczmarz-type methods \cite{YC19,ZJ19,XY21,BWW18,BBW18}, the extended version of Kaczmarz-type methods \cite{MN15,BW199}, and many others \cite{JS16,GL20,DH19,PP12,NW13}.  Kaczmarz's research also accelerated the development of column action iterative methods represented by the coordinate descent method \cite{D10}. See \cite{ZG20,W15,RT14,NS17,NN17,LX15,CH08,BW19}, etc.

Recently, Bai and Wu \cite{BWW18} proposed a new randomized row index selection strategy, which is aimed at grasping larger entries of the residual vector at each iteration, and constructed a greedy randomized Kaczmarz (GRK) method. They proved that the convergence of the GRK method is faster than that of the RK method. Due to its greedy selection strategy for row index, a large number of greedy versions of Kaczmarz work have been developed and studied.
At present, a lot of work is based on Kaczmarz's theory of orthogonal projection. In \cite{Popa12,PP12}, Constantin Popa gives the definition of oblique projection, which breaks the limitation of orthogonal projection. Therefore, in this paper, we propose a new descent direction based on the definition of oblique projection, which can guarantee the two entries of residual error to be zero during iteration, so as to accelerate convergence. Based on the row index selection rules of two representative randomized and non-randomized Kaczmarz-type methods -- the GRK method and the MWRK method, we propose two new Kaczmarz-type methods with oblique projection (KO-type) -- the GRKO method and the MWRKO method respectively, and their convergence is proved theoretically and numerically. We emphasize the efficiency of our proposed methods when the rows of the matrix $A$ are nearly linearly correlated, and find that Kaczmarz-type method based on orthogonal projection performed poorly when applied to this kind of matrices.

The organization of this paper is as follows. In Section 2, we introduce the KO-type method, and give its two lemmas. In Section 3, we propose the GRKO method and MWRKO method naturally and prove the convergence of the two methods. In Section 4, some numerical examples are provided to illustrate the efficiency of our new methods. Finally, some brief concluding remarks are described in Section 5.

In this paper, $\langle \cdot \rangle$ stands for the scalar product. $\|x\|$ is the Euclid norm of $x\in R^n$. For a given matrix $G=(g_{ij})\in R^{m\times n}$, $g_i^T$, $G^T$, $G^\dag$, $R(G)$, $N(G)$  ,$\|G\|_F$ and $\lambda_{min}(G)$, are used to denote the ith row, the transpose, the Moore-Penrose pseudoinverse \cite{Be74}, the range space, the null space, the Frobenius norm, and the smallest nonzero eigenvalue of $G$ respectively. $P_C(x)$ is the orthogonal projection of $x$ onto $C$, $\tilde{x}$ is any solution of the system (\ref{Ax=b}); $x^*=A^{\dag}b$ is the least-norm solution of the system (\ref{Ax=b}). Let $E_k$ denote the expected value conditonal on the first k iterations, that is,
$$E_k[\cdot]=E[\cdot|j_0,j_1,...,j_{k-1}],$$
where $j_s(s=0,1,...,k-1)$ is the column chosen at the sth iteration.
\section{Kaczmarz-type Method with Oblique Projection and its Lemmas}
The sets $H_i=\left\{x\in R^n,\langle a_i,x\rangle=b_i \right\}\ (i=1,2,\cdots,m)$ are the hyperplanes which associated to the $i$th equation of the system (\ref{Ax=b}) . To project the current iteration point $x^{(k)}$ to one of the hyperplanes, the oblique projection \cite{Popa12,PP12} can be expressed as follows:
\begin{align}
\label{oblique}
x^{(k+1)}=P_{H_{i}}^d(x^{(k)})=x^{(k)}-\frac{\langle a_i,x\rangle-b_i}{\langle d,a_i\rangle}d,
\end{align}
where $d\in R^n$ is a given direction. In Figure \ref{figure1}, $x^{(k+1)}$ is obtained by oblique projection of the current iteration point $x^{(k)}$ to the hyperplane $H_{i_{k+1}}$ along the direction $d$, i.e. $x^{(k+1)}=P_{H_{i_{k+1}}}^d(x^{(k)})$. $y^{(k+1)}$ is the iteration point obtained when the direction $d=a_{i_{k+1}}$, i.e. $y^{(k+1)}=P_{H_{i_{k+1}}}^{a_{i_{k+1}}}(x^{(k)})$. When the direction $d=a_i\,(i=mod(m,k)+1)$, it is the classic Kaczmarz method. However, when the hyperplanes are close to linear parallel, the Kaczmarz method based on orthogonal projection has a slow iteration speed.  In this paper, we propose  a new iteration direction $d=w^{(i_k)}=a_{i_{k+1}}-\frac{\langle a_{i_{k}}, a_{i_{k+1}}\rangle}{\|a_i\|^2}a_{i_{k}}$, to make the current iteration point approach to the intersection of two hyperplanes, i.e. $z^{(k+1)}=P_{H_{i_{k+1}}}^{w^{(i_{k})}}(x^{(k)})$.
\begin{figure}
  \centering
  \includegraphics[width=4.5in]{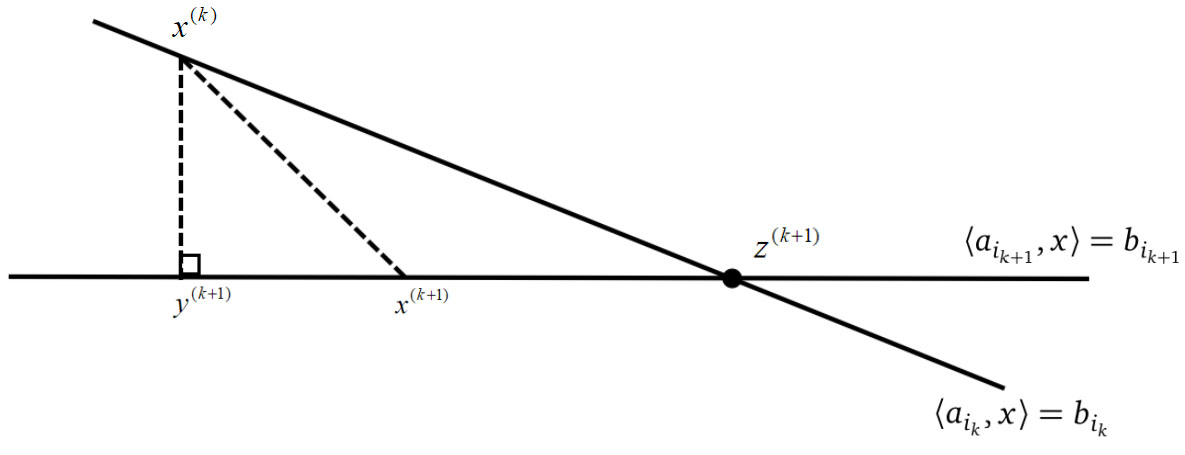}\\
  \caption{Oblique projection in different directions $d$}\label{figure1}
\end{figure}

The framework of KO-type mthod is given in Section 2.1.
\begin{algorithm}
\label{KO method}
  \leftline{\caption{Kaczmarz-type Method with Oblique Projection}}
  \begin{algorithmic}[1]
    \Require
      $A\in R^{m\times n}$, $b\in R^{m}$, $x^{(0)}\in R^n$, $K$, $\varepsilon>0$
    \State For $i=1:m$, $M(i)=\|a_i\|^2$
    \State Choose $i_{1}$ based on a certain selection rule
    \State Compute $x^{(1)}=x^{(0)}+\frac{b_{i_{1}}-\langle a_{i_{1}}, x^{(0)}\rangle}{M(i_{1})}a_{i_{1}}$
    \For {$k=1,2,\cdots, K$}
    \State Choose $i_{k+1}$ based on a certain selection rule
    \State Compute $D_{i_k}=\langle a_{i_{k}}, a_{i_{k+1}}\rangle$ \ and $r_{i_{k+1}}^{(k)}=b_{i_{k+1}}-\langle a_{i_{k+1}}, x^{(k)}\rangle$
    \State Compute $w^{(i_k)}=a_{i_{k+1}}-\frac{D_{i_k}}{M(i_k)}a_{i_{k}}$ \ and \ $h_{i_k}(=\|w^{(i_k)}\|^2)=M(i_{k+1})-\frac{D_{i_k}}{M(i_k)}D_{i_k}$
    \State $\alpha_{i_k}^{(k)}=\frac{r_{i_{k+1}}^{(k)}}{h_{i_k}}$ \ and $x^{(k+1)}=x^{(k)}+\alpha_{i_k}^{(k)} w^{({i_k})}$
    \EndFor
    \State Output $x^{(K+1)}$
  \end{algorithmic}
\end{algorithm}

  We will give two lemmas of KO-type method. The selection rule of its row index $i_{k+1}$ does not affect the lemmas.
\begin{lemma}
\label{lemma2}
For the Kaczmarz-type method with oblique projection, the residual satisfies the following equations:
\begin{equation}
r_{i_{k}}^{(k)}=0\quad(\forall k > 0),\label{rikk}
\end{equation}

\begin{equation}
 r_{i_{k-1}}^{(k)}=0\quad(\forall k > 1).\label{rik-1k}
\end{equation}
\end{lemma}
\begin{proof}
From the definition of the KO-type method, for $k=1$, we have
\begin{align*}
r^{(1)}_{i_{1}}&=b_{i_{1}}-\langle a_{i_{1}},x^{(1)}\rangle\\
&=b_{i_{1}}-\langle a_{i_{1}},x^{(0)}+\frac{b_{i_{1}}-\langle a_{i_{1}},x^{(0)}\rangle}{M(i_1)}a_{i_{1}}\rangle\\
&=0.
\end{align*}
For $k>1$, we have
\begin{align*}
r^{(k)}_{i_{k}}&=b_{i_k}-\langle a_{i_{k}},x^{(k)} \rangle\\
&=b_{i_{k}}-\langle a_{i_{k}},x^{(k-1)}+\alpha^{(k-1)}_{i_{k-1}}w^{(i_{k-1})} \rangle\\
&=b_{i_k}-\langle a_{i_{k}},x^{(k-1)} \rangle-\alpha^{(k-1)}_{i_{k-1}}\langle a_{i_k},w^{(i_{k-1})}\rangle\\
&=r^{(k-1)}_{i_k}-\frac{r^{(k-1)}_{i_{k}}}{h_{i_{k-1}}}\langle a_{i_{k}},w^{(i_{k-1})}\rangle\\
&=r^{(k-1)}_{i_k}-r^{(k-1)}_{i_k}\\
&=0.
\end{align*}
The fifth equality holds due to $\langle a_{i_{k}},w^{(i_{k-1})}\rangle=||a_{i_{k}}||^2-\frac{\langle a_{i_{k-1}},a_{i_{k}}\rangle^2}{||a_{i_{k-1}}||^2}=h_{i_{k-1}}$.
Thus, the equation (\ref{rikk}) holds.

Since $k>1$, $$x^{(k)}=x^{(k-1)}+\alpha^{(k)}_{i_{k-1}}w^{(i_{k-1})}.$$\\
We get $$(b-Ax^{(k)})_{i_{k-1}}=(b-Ax^{(k-1)})_{i_{k-1}}-(A\alpha^{(k)}_{i_{k-1}}w^{(i_{k-1})})_{i_{k-1}},$$
that is,
\begin{align*}
r^{(k)}_{i_{k-1}}&=r_{i_{k-1}}^{(k-1)}-\alpha^{(k)}_{i_{k-1}}\langle a_{i_{k-1}},w^{(i_{k-1})}\rangle\\
&=\alpha^{(k)}_{i_{k-1}}\langle a_{i_{k-1}},a_{i_{k}}-\frac{\langle a_{i_{k-1}},a_{i_{k}}\rangle}{||a_{i_{k-1}}||^2}a_{i_{k-1}}\rangle\\
&=0.
\end{align*}
The second equality holds due to the equation (\ref{rikk}). Thus, the equation (\ref{rik-1k}) holds.
\end{proof}
\begin{lemma}
The iteration sequence $\left\{x^{(k)}\right\}^\infty_{k=0}$ generated by the Kaczmarz-type method with oblique projection, satisifies the following equations:
\begin{equation}
||x^{(k+1)}-\tilde{x}||^2=||x^{(k)}-\tilde{x}||^2-||x^{(k+1)}-x^{(k)}||^2\quad(\forall k \geq 0),\label{xk+1-x}
\end{equation}
where $\tilde{x}$ is an arbitrary solution of the system (\ref{Ax=b}). Especially, when $P_{N(A)}(x^{(0)})=P_{N(A)}(\tilde{x})$, $x^{(k)}-\tilde{x}\in R(A^T)$.
\end{lemma}
\begin{proof}
For $k=0$, we have
\begin{align*}
\langle a_{i_{1}},x^{(1)}-\tilde{x}\rangle&=\langle a_{i_{1}},x^{(0)}-\tilde{x}+\frac{b_{i_{1}}-\langle a_{i_{1}},x^{(0)}\rangle}{M(i_1)}a_{i_{1}}\rangle\\
&=\langle a_{i_{1}},x^{(0)}\rangle-b_{i_{1}}+\langle a_{i_{1}},\frac{b_{i_{1}}-\langle a_{i_{1}},x^{(0)}\rangle}{M(i_1)}a_{i_{1}}\rangle\\
&=0,
\end{align*}
which shows that $x^{(1)}-\tilde{x}$ is orthogonal to $a_{i_{1}}$. Therefore, we know
\begin{align*}
(x^{(1)}-x^{(0)})^T(x^{(1)}-\tilde{x})=0.
\end{align*}
It follows that
\begin{align}
||x^{(1)}-\tilde{x}||^2=||x^{(0)}-\tilde{x}||^2-||x^{(1)}-x^{(0)}||^2.\label{x1-x0}
\end{align}
For $k>0$, we have
\begin{align*}
\langle w^{(i_{k})},x^{(k+1)}-\tilde{x}\rangle&=\langle w^{(i_{k})},x^{(k)}-\tilde{x}+\alpha^{(k)}_{i_{k}}w^{(i_{k})}\rangle\\
&=\langle a_{i_{k+1}}-\frac{D_{i_{k}}}{M(i_{k})}a_{i_{k}},x^{(k)}-\tilde{x}\rangle+\langle w^{(i_k)},\frac{r_{i_{k+1}}^{(k)}}{h_{i_{k}}}w^{(i_{k})}\rangle\\
&=-r^{(k)}_{i_{k+1}}+\frac{D_{i_{k}}}{M(i_{k})}r_{i_{k}}^{(k)}+r_{i_{k+1}}^{(k)}\\
&=0.
\end{align*}
The third  and last equalities hold due to $h_{i_{k}}=\|w^{(i_{k})}\|^2$, and the equation (\ref{rikk}) respectively. Thus we get that $x^{(k+1)}-\tilde{x}$ is orthogonal to $w^{(i_{k})}$. Therefore, we get that
\begin{align*}
(x^{(k+1)}-x^{(k)})^T(x^{(k+1)}-\tilde{x})=0.
\end{align*}
It follows that
\begin{align}
||x^{(k+1)}-\tilde{x}||^2=||x^{(k)}-\tilde{x}||^2-||x^{(k+1)}-x^{(k)}||^2\quad(\forall k > 0).\label{xk+1-xk}
\end{align}

Thus, from the above proof, the equation (\ref{xk+1-x}) holds.

According to the iterative formula
$$
\left\{
\begin{aligned}
x^{(1)}&=x^{(0)}+\frac{b_{i_{1}}-\langle a_{i_{1}},x^{(0)}\rangle}{M(i_1)}a_{i_{1}},\\
x^{(k+1)}&=x^{(k)}+\alpha^{(k)}_{i_{k}}w^{(i_{k})}\quad(\forall k>0),\\
\end{aligned}
\right.
$$
we can get $P_{N(A)}(x^{(k)})=P_{N(A)}(x^{(k-1)})=...=P_{N(A)}(x^{(0)}),$
and by the fact that $P_{N(A)}(x^{(0)})=P_{N(A)}(\tilde{x})$, we can deduce that $x^{(k)}-\tilde{x} \in R(A^T)$.\\
\end{proof}

\section{Greedy Randomized and Maximal Weighted Residual Kaczmarz methods with Oblique Projection}
In this section, we combine the oblique projection with the GRK method \cite{BWW18} and the MWRK method \cite{S77} to obtain the GRKO method and the MWRKO method, and prove their convergence. Theoretical results show that the KO-type method can accelerate the convergence when there are suitable row index selection strategies.
\subsection{Greedy Randomized Kaczmarz Method with Oblique projection}
The core of the GRK method \cite{BWW18} is a new probability criterion, which can grasp the large items of the residual vector in each iteration, and randomly select the item with probability in proportion to the retained residual norm. Theories and experiments prove that it can speed up convergence speed. This paper uses its the row index selection rule in combination with the KO method to obtain the GRKO method, and the algorithm is as follows:
\begin{breakablealgorithm}
\caption{Greedy Randomized Kaczmarz Method with Oblique Projection}
  \label{GRKO}
  \begin{algorithmic}[1]
    \Require
      $A\in R^{m\times n}$, $b\in R^{m}$, $x^{(0)}\in R^n$, $K$, $\varepsilon>0$
    \State For $i=1:m$, $M(i)=\|a_i\|^2$
    \State Randomly select $i_1$,and compute $x^{(1)}=x^{(0)}+\frac{b_{i_1}-\langle a_{i_1},x^{(0)} \rangle}{M(i_1)}$$a_{i_1}$
\State \textbf{for} $k=1,2,...,K-1$ \textbf{do}
\State \ \ \ \ Compute $\varepsilon _k=\frac{1}{2}\left(\frac{1}{||b-Ax^{(k)}||^2}\mathop{max}\limits_{1\leq i_{k+1} \leq m}\left\{ \frac{|b_{i_{k+1}}-\langle a_{i_{k+1}},x^{(k)} \rangle|^2 }{||a_{i_{k+1}}||^2} \right\}+\frac{1}{||A||^2_F}\right)$
\State \ \ \ \ Determine the index set of positive integers
 $$\mathcal{U}_k=\left\{ i_{k+1}| |b_{i_{k+1}}-\langle a_{i_{k+1}},x^{(k)}\rangle|^2 \ge \varepsilon _k||b-Ax^{(k)}||^2||a_{i_k+1}||^2 \right\} $$
\State \ \ \ \ Compute the ith entry $\tilde{r}_i^{(k)}$ of the vector $\tilde{r}^{(k)}$ according to
$$\tilde{r}_i^{(k)}=
\begin{cases}
b_i-\langle a_i,x^{(k)}\rangle,  & \mbox{if }i \in \mathcal{U}_k\\
 0 & \mbox{otherwise }
\end {cases}$$
\State \ \ \ \ \  Select $i_{k+1}  \in \mathcal{U}_k$  with probability $Pr(row=i_{k+1})= \frac{|\tilde{r}_{i_{k+1}}^{(k)}|^2}{||\tilde{r}^{(k)}||^2}$
\State  \ \ \ \ \ Compute $D_{i_{k}}=\langle a_{i_k},a_{i_{k+1}}\rangle$
\State \ \ \ \ Compute $w^{(i_k)}=a_{i_{k+1}}-\frac{D_{i_k}}{M(i_k)}a_{i_{k}}$  and  $$h_{i_k}(=\|w^{(i_k)}\|^2)=M(i_{k+1})-\frac{D_{i_k}}{M(i_k)}D_{i_k}$$
\State \ \ \ \ \ $\alpha_{i_k}^{(k)}=\frac{\tilde{r}_{i_{k+1}}^{(k)}}{h_{i_k}}\left( =\frac{r^{(k)}_{i_{k+1}}}{h_{i_{k}}}\right)$ \ and $x^{(k+1)}=x^{(k)}+\alpha_{i_k}^{(k)} w^{({i_k})}$
\State \textbf{end for}
\State Output $x^{(K)}$
\end{algorithmic}
\end{breakablealgorithm}
The convergence of the GRKO method is provided as follows.
\begin{theorem}
Consider the consistent linear system (\ref{Ax=b}), where the coefficient matrix $A\in R^{m\times n}$, $b \in R^m$. Let $x^{(0)}\in R^{n}$ be an arbitrary initial approximation , $\tilde{x}$ is a solution of system (\ref{Ax=b}) such that $P_{N(A)}(\tilde{x})=P_{N(A)}(x^{(0)})$. Then the iteration sequence$ \left\{x^{(k)}\right\}^{\infty}_{k=1}$ generated by the GRKO method obeys
\begin{align}
\label{E}
E||x^{(k)}-\tilde{x}||^2\leq \mathop{\Pi}\limits_{s=0}^{k-1}\zeta_s||x^{(0)}-\tilde{x}||^2.
\end{align}
where $\zeta_0=1-\frac{(\lambda_{min}(A^TA))}{m||A||^2_F},$
$\zeta_1=1-\frac{1}{2}(\frac{1}{\gamma_1}||A||^2_F+1)\frac{\lambda_{min}(A^TA)}{\Delta\cdot||A||^2_F},$
$\zeta_k=1-\frac{1}{2}(\frac{1}{\gamma_2}||A||^2_F+1)\frac{\lambda_{min}(A^TA)}{\Delta\cdot||A||^2_F}\quad(\forall k > 1)$, which
\begin{align}
\label{gamma1}
\gamma_1=\mathop{max}\limits_{\substack{1\leq i\leq m}}\sum\limits_{\substack{s=1\\s\neq i}}^{m}||a_s||^2,
\end{align}
\begin{align}
\label{gamma2}
\gamma_2=\mathop{max}\limits_{\substack{1\leq i,j\leq m\\i\neq j}}\sum\limits_{\substack{s=1\\s\neq i,j}}^{m}||a_s||^2,
\end{align}
\begin{align}
\label{delta}
\Delta=\mathop{max}\limits_{\substack{j\neq k}}sin\langle a_j,a_k\rangle^2(\in(0,1]).
\end{align}

 In addition, if $x^{(0)}\in R(A^T)$, the sequence $ \left\{x^{(k)}\right\}^{\infty}_{k=1}$ converges to the least-norm solution of the system (\ref{Ax=b}), i.e. $\lim\limits_{k\rightarrow\infty}x^{(k)}=x^*=A^\dag b$.
\end{theorem}
\begin{proof}
When $k=1$, we can get
\begin{equation}
\begin{aligned}
\label{varepsilon1}
\varepsilon_1||A||^2_F&=\frac{\mathop{max}\limits_{1\leq i_{2} \leq m} \left\{\frac{|b_{i_{2}}-\langle a_{i_{2}},x^{(1)}\rangle|^2}{||a_{i_{2}}||^2}\right\}}{2\sum\limits_{i_{2}=1}^m\frac{||a_{i_{2}}||^2}{||A||^2_F}.\frac{|b_{i_{2}}-\langle a_{i_{2}},x^{(1)}\rangle|^2}{||a_{i_{2}}||^2}}+\frac{1}{2}\\
&=\frac{\mathop{max}\limits_{1\leq i_{2} \leq m} \left\{\frac{|b_{i_{2}}-\langle a_{i_{2}},x^{(1)}\rangle|^2}{||a_{i_{2}}||^2}\right\}}{2\sum\limits_{\substack{i_{2}=1\\i_{2}\neq i_{1}}}^m\frac{||a_{i_{2}}||^2}{||A||^2_F}.\frac{|b_{i_{2}}-\langle a_{i_{2}},x^{(1)}\rangle|^2}{||a_{i_{2}}||^2}}+\frac{1}{2}\\
&\geq \frac{1}{2} \left( \frac{||A||^2_F}{\sum\limits_{\substack{i_{2}=1\\i_{2}\neq i_{1}}}^m||a_{i_{2}}||^2}+1\right)\\
&\geq \frac{1}{2}\left( \frac{1}{\gamma_1}||A||^2_F+1\right).
\end{aligned}
\end{equation}
The second equality holds due to the equation (\ref{rikk}).

When $k>1$, we get
\begin{equation}
\begin{aligned}
\label{varepsilonk}
\varepsilon_k||A||^2_F&=\frac{\mathop{max}\limits_{1\leq i_{k+1} \leq m} \left(\frac{|b_{i_{k+1}}-\langle a_{i_{k+1}},x^{(k)}\rangle|^2}{||a_{i_{k+1}}||^2}\right)}{2\sum\limits_{i_{k+1}=1}^m\frac{||a_{i_{k+1}}||^2}{||A||^2_F}.\frac{|b_{i_{k+1}}-\langle a_{i_{k+1}},x^{(k)}\rangle|^2}{||a_{i_{k+1}}||^2}}+\frac{1}{2}\\
&=\frac{\mathop{max}\limits_{1\leq i_{k+1} \leq m}\left(\frac{|b_{i_{k+1}}-\langle a_{i_{k+1}},x^{(k)}\rangle|^2}{||a_{i_{k+1}}||^2}\right)}{2 \sum\limits_{\substack{i_{k+1} = 1\\ i_{k+1} \neq i_{k}, i_{k-1}}}^{m} \frac{||a_{i_{k+1}}||^2}{||A||^2_F}.\frac{|b_{i_{k+1}}-\langle a_{i_{k+1}},x^{(k)}\rangle|^2}{||a_{i_{k+1}}||^2}}+\frac{1}{2}\\
&\geq \frac{1}{2} \left(\ \frac{||A||^2_F}{\sum\limits_{\substack{i_{k+1}=1\\i_{k+1}\neq i_{k}, i_{k-1}}}^{m}||a_{i_{k+1}}||^2}+1\right)\\
&\geq \frac{1}{2}\left( \frac{1}{\gamma_2}||A||^2_F+1\right).
\end{aligned}
\end{equation}
The second equality holds due to the equation (\ref{rikk}) and the equation (\ref{rik-1k}).

Under the GRKO method, Lemma 2.2 still holds, so we can take the full expectation on both sides of the equation (\ref{xk+1-x}), and get that for $k=0$,
\begin{equation}
\begin{aligned}
\label{E1}
E||x^{(1)}-\tilde{x}||^2&=||x^{(0)}-\tilde{x}||^2-E||x^{(1)}-x^{(0)}||^2\\
&=||x^{(0)}-\tilde{x}||^2-\frac{1}{m}\sum\limits_{i_{1}=1}^m||\frac{b_{i_{1}}-\langle a_{i_{1}},x^{(0)}\rangle}{M(i_1)}a_{i_{1}}||^2\\
&\leq||x^{(0)}-\tilde{x}||^2-\frac{1}{m}\frac{||b-Ax^{(0)}||^2}{||A||_F^2}\\
&\leq \left(1-\frac{\lambda_{min}(A^TA)}{m||A||^2_F}\right)||x^{(0)}-\tilde{x}||^2\\
&=\zeta_0||x^{(0)}-\tilde{x}||^2,
\end{aligned}
\end{equation}
and for $k>0$,
\begin{equation}
\begin{aligned}
\label{Ek}
\mathrm{E}_k||x^{(k+1)}-\tilde{x}||^2&=||x^{(k)}-\tilde{x}||^2-\mathrm{E}_k||x^{(k+1)}-x^{(k)}||^2\\
&=||x^{(k)}-\tilde{x}||^2-\sum\limits_{i_{k+1} \in \mathcal{U}_k}\frac{|b_{i_{k+1}}-\langle a_{i_{k+1}},x^{(k)}\rangle |^2}{\sum\limits_{i_{k+1} \in \mathcal{U}_k}|b_{i_{k+1}}-\langle a_{i_{k+1}},x^{(k)}\rangle|^2}.\frac{|r^{(k)}_{i_{k+1}}|^2}{||w^{(i_k)}||^2}\\
&\leq||x^{(k)}-\tilde{x}||^2-\sum\limits_{i_{k+1} \in \mathcal{U}_k}\frac{|b_{i_{k+1}}-\langle a_{i_{k+1}},x^{(k)}\rangle |^2}{\sum\limits_{i_{k+1} \in \mathcal{U}_k}|b_{i_{k+1}}-\langle a_{i_{k+1}},x^{(k)}\rangle|^2}.\frac{|r^{(k)}_{i_{k+1}}|^2}{\Delta\cdot||a_{i_{k+1}}||^2}\\
&\leq||x^{(k)}-\tilde{x}||^2-\frac{\varepsilon_k}{\Delta}||b-Ax^{(k)}||^2\\
&=||x^{(k)}-\tilde{x}||^2-\frac{\varepsilon_k}{\Delta}||A(\tilde{x}-x^{(k)})||^2\\
&\leq(1-\frac{\varepsilon_k\lambda_{min}(A^TA)}{\Delta})||x^{(k)}-\tilde{x}||^2.
\end{aligned}
\end{equation}
The first inequality of the equation (\ref{E1}) is achieved with the use of the fact that $\frac{|b_1|}{|a_1|}+\frac{|b_2|}{|a_2|}\geq\frac{|b_1|+|b_2|}{|a_1|+|a_2|}$ (if $|a_1|>0$, $|a_2|>0$), and the first inequality of the equation (\ref{Ek}) is achieved with the use of the fact that $||w_{i_{k}}||^2=||a_{i_{k+1}}||^2-\frac{\langle a_{i_k},a_{i_{k+1}}\rangle^2}{||a_{i_{k+1}}||^2}=sin\langle a_{i_{k}},a_{i_{k+1}}\rangle^2||a_{i_{k+1}}||^2\leq\Delta\cdot||a_{i_{k+1}}||^2$, and the second inequality of the equation (\ref{Ek}) is achieved with the use of the definition of $\mathcal{U}_k$ which lead to
$$|b_{i_{k+1}}-\langle a_{i_{k+1}},x^{(k)}\rangle|^2\geq\varepsilon_k||b-Ax^{(k)}||^2||a_{i_{k+1}}||^2,{\forall}i_{k+1} \in \mathcal{U}_k .$$
Here in the last inequalities of the equation (\ref{E1}) and (\ref{Ek}), we have used the estimate
$||Au||^2_2\geq\lambda_{min}(A^TA)||u||^2$, which holds true for any $u \in C^n$ belonging to the column space of $A^T$.
According to the lemma 2.2, it holds.

By making use of the equation (\ref{varepsilon1}), (\ref{varepsilonk}) and (\ref{Ek}), we get
\begin{align*}
E_1||x^{(2)}-\tilde{x}||^2&\leq \left[1-\frac{1}{2}(\frac{1}{\gamma_1}||A||^2_F+1)\frac{\lambda_{min}(A^TA)}{\Delta\cdot||A||^2_F}\right]||x^{(1)}-\tilde{x}||^2\\
&=\zeta_1||x^{(1)}-\tilde{x}||^2,
\end{align*}
\begin{align*}
E_k||x^{(k+1)}-\tilde{x}||^2&\leq \left[1-\frac{1}{2}(\frac{1}{\gamma_2}||A||^2_F+1)\frac{\lambda_{min}(A^TA)}{\Delta\cdot||A||^2_F}\right]||x^{(k)}-\tilde{x}||^2 \\
&=\zeta_k||x^{(k)}-\tilde{x}||^2\quad(\forall k > 1).
\end{align*}

Finally, by recursion and taking the full expectation , the equation (\ref{E}) holds.
\end{proof}
{\bf Remark 1.} \ In the GRKO method, $h_{i_{k}}$ is not zero. Suppose $h_{i_{k}}=0$, which means $\exists \lambda >0$,  $\lambda a_{i_{k}}= a_{i_{k+1}}$. Due to the system is consistent, it holds $\langle a_{i_{k+1}},x^{*}\rangle =\lambda\langle a_{i_{k}},x^{*}\rangle=\lambda b_{i_{k}}=b_{i_{k+1}} $. According to the equation (\ref{rikk}), it holds $r^{(k)}_{i_{k+1}}=\lambda r^{(k)}_{i_{k}}=0$. From step 5 of Algorithm 3.1, we can konw that such index $i_{k+1}$ will not be selected.

{\bf Remark 2.} \ Set $\tilde{\zeta_k}=1-\frac{1}{2}(\frac{1}{\gamma_1}||A||^2_F+1)\frac{\lambda_{min}(A^TA)}{||A||^2_F}\quad(\forall k > 0)$, and the convergence of GRK method in \cite{BWW18} meets:
$$E_k\|x^{(k+1)}-x^*\|^2\leq\tilde{\zeta_k}\|x^{(k)}-x^*\|^2.$$
Obviously, $\zeta_1\leq\tilde{\zeta_1},\zeta_k<\tilde{\zeta_k}\quad(\forall k > 1)$ is satisfied, so the convergence speed of GRKO method is faster than GRK method.

\subsection{Maximal Weighted Residual Kaczmarz Method with Oblique Projection}
The selection strategy for the index $i_k$ used in the maximal weighted residual Kaczmarz (MWRK) method \cite{S77} is: Set
$$i_{k}=\mathop{arg\max}\limits_{i\in \left\{1,2,\cdots,m\right\}}\frac{|a_{i}^{T}x^{(k)}-b_{i}|}{\|a_i\|}.$$
McCormick proved the exponential convergence of the MWRK method. In \cite{DH19}, a new convergence conclusion of the MWRK method is given. We use its row index selection rule combined with KO-type method to obtain MWRKO method, and the algorithm is as follows:
\begin{breakablealgorithm}
  \caption{Maximal Weighted Residual Kaczmarz Method with Oblique Projection (MWRKO)}
  \label{MWRKO}
  \begin{algorithmic}[1]
    \Require $A\in R^{m\times n}$, $b\in R^{m}$, $x^{(0)}\in R^n$, $K$, $\varepsilon>0$
    \State For $i=1:m$, $M(i)=\|a_i\|^2$
    \State Compute $i_1=\mathop{arg\max}\limits_{i\in \left\{1,2,\cdots,m\right\}}\frac{|a_{i}^{T}x^{(0)}-b_{i}|}{\|a_i\|}$, and $x^{(1)}=x^{(0)}+\frac{b_{i_{1}}-\langle a_{i_{1}}, x^{(0)}\rangle}{M(i_{1})}a_{i_{1}}$
    \For {$k=1,2,\cdots, K$}
    \State Compute $i_{k+1}=\mathop{arg\max}\limits_{i\in \left\{1,2,\cdots,m\right\}}\frac{|a_{i}^{T}x^{(k)}-b_{i}|}{\|a_i\|}$
    \State Compute $D_{i_k}=\langle a_{i_{k}}, a_{i_{k+1}}\rangle$ \ and $r_{i_{k+1}}^{(k)}=b_{i_{k+1}}-\langle a_{i_{k+1}}, x^{(k)}\rangle$
    \State Compute $w^{(i_k)}=a_{i_{k+1}}-\frac{D_{i_k}}{M(i_k)}a_{i_{k}}$ \ and \ $h_{i_k}(=\|w^{(i_k)}\|^2)=M(i_{k+1})-\frac{D_{i_k}}{M(i_k)}D_{i_k}$
    \State $\alpha_{i_k}^{(k)}=\frac{r_{i_{k+1}}^{(k)}}{h_{i_k}}$ \ and $x^{(k+1)}=x^{(k)}+\alpha_{i_k}^{(k)} w^{({i_k})}$
    \EndFor
    \State Output $x^{(K+1)}$
  \end{algorithmic}
\end{breakablealgorithm}

The convergence of the MWRKO method is provided as follows.

\begin{theorem}
 Consider the consistent linear system (\ref{Ax=b}), where the coefficient matrix $A\in R^{m\times n}$, $b \in R^m$. Let $x^{(0)}\in R^{n}$ be an arbitrary initial approximation , $\tilde{x}$ is a solution of system (\ref{Ax=b}) such that $P_{N(A)}(\tilde{x})=P_{N(A)}(x^{(0)})$. Then the iteration sequence$ \left\{x^{(k)}\right\}^{\infty}_{k=1}$ generated by the MWRKO method obeys
 \begin{align}
\label{Tf}
||x^{(k)}-\tilde{x}||^2\leq \mathop{\Pi}\limits_{s=0}^{k-1}\rho_s||x^{(0)}-\tilde{x}||^2,
\end{align}
where $\rho_0=1-\frac{\lambda_{min}(A^TA)}{\|A\|^2_F},$
$\rho_1=1-\frac{\lambda_{min}(A^TA)}{\Delta\cdot\gamma_1},$
$\rho_k=1-\frac{\lambda_{min}(A^TA)}{\Delta\cdot\gamma_2}\quad(\forall k > 1),$
which $\gamma_1$, $\gamma_2$ and $\Delta$ are defined by equations (\ref{gamma1}), (\ref{gamma2}) and (\ref{delta}) respectively.

In addition, if $x^{(0)}\in R(A^T)$, the sequence $ \left\{x^{(k)}\right\}^{\infty}_{k=1}$ converges to the least-norm solution of the system (\ref{Ax=b}), i.e. $\lim\limits_{k\rightarrow\infty}x^{(k)}=x^*=A^\dag b$.
\end{theorem}
\begin{proof}
Under the MWRKO method, Lemma 2.2 still holds. For $k=1$, we have
\begin{equation}
\begin{aligned}
\label{T1}
\|x^{(1)}-\tilde{x}\|^2&=\|x^{(0)}-\tilde{x}\|^2-\|x^{(1)}-x^{(0)}\|^2\\
&=\|x^{(0)}-\tilde{x}\|^2-\frac{|b_{i_{1}}-\langle a_{i_{1}},x^{(0)}\rangle|^2}{M(i_1)}\\
&=\|x^{(0)}-\tilde{x}\|^2-\frac{|b_{i_{1}}-\langle a_{i_{1}},x^{(0)}\rangle|^2}{M(i_1)}\cdot\frac{\|b-Ax^{(0)}\|^2}{\sum\limits_{\substack{i=1}}^{m}\frac{|b_{i}-\langle a_{i},x^{(0)}\rangle|^2}{M(i)}\cdot  M(i)}\\
&\leq\|x^{(0)}-\tilde{x}\|^2-\frac{\|A(\tilde{x}-x^{(0)})\|^2}{\|A\|^2_F}\\
&\leq\|x^{(0)}-\tilde{x}\|^2-\frac{\lambda_{min}(A^TA)}{\|A\|^2_F}\|x^{(0)}-\tilde{x}\|^2\\
&=\left(1-\frac{\lambda_{min}(A^TA)}{\|A\|^2_F}\right)\|x^{(0)}-x^*\|^2\\
&=\rho_0\|x^{(0)}-x^*\|^2.
\end{aligned}
\end{equation}
For $k=1$,we have
\begin{equation}
\begin{aligned}
\label{T2}
\|x^{(2)}-\tilde{x}\|^2&=\|x^{(1)}-\tilde{x}\|^2-\|x^{(2)}-x^{(1)}\|^2\\
&=\|x^{(1)}-\tilde{x}\|^2-\frac{|b_{i_{2}}-\langle a_{i_{2}},x^{(1)}\rangle|^2}{\|w^{(i_1)}\|^2}\\
&\leq\|x^{(1)}-\tilde{x}\|^2-\frac{|b_{i_{2}}-\langle a_{i_{2}},x^{(1)}\rangle|^2}{\Delta\cdot M(i_2)}\cdot\frac{\|b-Ax^{(1)}\|^2}{\sum\limits_{\substack{i=1,i\neq i_1}}^{m}\frac{|b_{i}-\langle a_{i},x^{(1)}\rangle|^2}{M(i)}\cdot  M(i)}\\
&\leq\|x^{(1)}-\tilde{x}\|^2-\frac{\|A(\tilde{x}-x^{(1)})\|^2}{\Delta\cdot\gamma_1}\\
&\leq\|x^{(1)}-\tilde{x}\|^2-\frac{\lambda_{min}(A^TA)}{\Delta\cdot\gamma_1}\|x^{(1)}-\tilde{x}\|^2\\
&=\left(1-\frac{\lambda_{min}(A^TA)}{\Delta\cdot\gamma_1}\right)\|x^{(1)}-\tilde{x}\|^2\\
&=\rho_1\|x^{(1)}-\tilde{x}\|^2.
\end{aligned}
\end{equation}
For $k>1$,we have
\begin{equation}
\begin{aligned}
\label{T3}
\|x^{(k+1)}-\tilde{x}\|^2&=\|x^{(k)}-\tilde{x}\|^2-\|x^{(k+1)}-x^{(k)}\|^2\\
&=\|x^{(k)}-\tilde{x}\|^2-\frac{|b_{i_{k+1}}-\langle a_{i_{k+1}},x^{(k)}\rangle|^2}{\|w^{(i_{k})}\|^2}\\
&\leq\|x^{(k)}-\tilde{x}\|^2-\frac{|b_{i_{k+1}}-\langle a_{i_{k+1}},x^{(k)}\rangle|^2}{\Delta\cdot M(i_{k+1})}\cdot\frac{\|b-Ax^{(k)}\|^2}{\sum\limits_{\substack{i=1,i\neq i_k,i_{k-1}}}^{m}\frac{|b_{i}-\langle a_{i},x^{(k)}\rangle|^2}{M(i)}\cdot  M(i)}\\
&\leq\|x^{(k)}-\tilde{x}\|^2-\frac{\|A(\tilde{x}-x^{(k)})\|^2}{\Delta\cdot\gamma_2}\\
&\leq\|x^{(k)}-\tilde{x}\|^2-\frac{\lambda_{min}(A^TA)}{\Delta\cdot\gamma_2}\|x^{(k)}-\tilde{x}\|^2\\
&=\left(1-\frac{\lambda_{min}(A^TA)}{\Delta\cdot\gamma_2}\right)\|x^{(k)}-\tilde{x}\|^2\\
&=\rho_k\|x^{(k)}-\tilde{x}\|^2.
\end{aligned}
\end{equation}
Here in the last inequalities of the equation (\ref{T1}), (\ref{T2}) and (\ref{T3}), we have used the estimate
$$||Au||^2_2\geq\lambda_{min}(A^TA)||u||^2,$$which holds true for any $u \in C^n$belonging to the column space of $A^T$.
According to the lemma 2.2, it holds.
From the equation (\ref{T1}), (\ref{T2}) and (\ref{T3}), the equation (\ref{Tf}) holds.
\end{proof}
{\bf Remark 3.}When multiple indicators $i_{k+1}$ are met in Step 2 of Algorithm 3.2 in the iterative process, we randomly select any one of them.

{\bf Remark 4.} In the MWRKO method, the reason of $h_{i_{k}}\neq 0$ is similar to Remark 1.

{\bf Remark 5.} \ Set $\tilde{\rho}_0=1-\frac{\lambda_{min}(A^TA)}{\|A\|^2_F}$,$\tilde{\rho}_k=1-\frac{\lambda_{min}(A^TA)}{\gamma_1}\quad(\forall k > 0),$
 and the convergence of MWRK method in \cite{DH19} meets:
$$||x^{(k)}-x^*||^2\leq \mathop{\Pi}\limits_{s=0}^{k-1}\tilde{\rho_s}||x^{(0)}-x^*||^2,$$
Obviously, $\rho_k<\tilde{\rho_k}\quad(\forall k > 1)$, $\rho_1\leq\tilde{\rho_1}$ and $\rho_0=\tilde{\rho_0}$ , so the convergence speed of MWRKO method is faster than MWRK method. Note that $\tilde{\rho}_k<\tilde{\zeta}_k$, $\rho_k<\zeta_k\quad(\forall k > 0)$, that is
$V_{MWRK}<V_{MWRKO}$, $V_{GRK}<V_{GRKO}$, $V_{GRK}<V_{MWRK}$, $V_{GRKO}<V_{MWRKO}$, where $V$ represents the convergence speed.
\section{Numerical Experiments}
In this section, some numerical examples are provided to illustrate the effectiveness of the greedy randomized Kaczmarz (GRK) method, the greedy randomized Kaczmarz method with oblique projection (GRKO), the maximal weighted residual Kaczmarz method (MWRK), and the maximal weighted residual Kaczmarz method (MWRKO) . All experiments are carried out using MATLAB (version R2019b) on a personal computer with 1.60 GHz central processing unit (Intel(R) Core(TM) i5-10210U CPU), 8.00 GB memory, and Windows operating system (64 bit Windows 10).

In our implementations, the right vector $b=Ax^*$ such that the exact solution $x^*\in R^n$ is a vector generated by the $rand$ function. Define the relative residual error (RRE) at the $k$th iteration as follows:
$$\text{RRE}=\frac{\|b-Ax^{(k)}\|^2}{\|b\|^2}.$$
The initial point $x^{(0)}\in R^n$ is set to be a zero vector, and the iterations are terminated once the relative solution error satisfies $\text{RRE}<\omega$ or the number of iteration steps exceeds 100,000. If the number of iteration steps exceeds 100,000, it is denoted as "-".

We will compare the numerical performance of these methods in terms of the number of iteration steps (denoted as "IT") and the computing time in seconds (denoted as "CPU"). Here the CPU and IT mean the arithmetical averages of the elapsed running
times and the required iteration steps with respect to 50 trials repeated runs of the corresponding method.
\subsection{Experiments for Random Matrix Collection in $[0,1]$}
The random matrix collection in $[0,1]$ is randomly generated by using the MATLAB function $rand$, and the numerical results are reported in Tables 1-2 and Figures 2-3. In this subsection, we let $\omega=0.5\times10^{-8}$. According to the characteristics of the matrix generated by  MATLAB function $rand$, Table 1 and Table 2 are  the
experiments for the overdetermined consistent linear systems, underdetermined consistent linear systems respectively. Under the premise of convergence, all methods can find the unique least Euclidean norm solution $x^*$.

From Table 1 and Figure 2, we can see that when the linear system is overdetermined, with the increase of $m$, the IT of all methods decreases, but the CPU shows an increasing trend. Our new methods -- the GRKO method and the MWRKO method, perform better than the GRK method and the MWRK method respectively in both iteration steps and running time. Among the four methods, the MWRKO method performs best. From Table 2 and Figure 3, we can see that in the case of underdetermined linear system, with the increase of $m$, the IT and CPU of all methods decrease.

In this group of experiments, whether it is an overdetermined or underdetermined linear system, whether in terms of the IT or CPU, the GRKO method and the MWRKO method perform very well compared with the GRK method and the MWRK method. These experimental phenomena are consistent with the theoretical convergence conclusions we got.
\begin{table}[H]
\label{table1}
\centering
\caption{IT and CPU of GRK, GRKO, MWRK and MWRKO for $m \times n$ matrices $A$ with $n = 500$ and different $m$ when the consistent linear system is overdetermined}
\renewcommand\arraystretch{0.9}
\begin{tabular}{lllllllllllll}
\cline{1-9}
\multicolumn{1}{c}{\multirow{2}{*}{m}} & \multicolumn{4}{c|}{IT}                           & \multicolumn{4}{c}{CPU}           &  &  &  &  \\ \cline{2-9}
\multicolumn{1}{c}{}                   & GRK   & GRKO & MWRK  & \multicolumn{1}{l|}{MWRKO} & GRK    & GRKO   & MWRK   & MWRKO  &  &  &  &  \\ \cline{1-9}
1000                                   & 12072 & 2105 & 11265 & \multicolumn{1}{l|}{1913}  & 1.2824 & 0.2099 & 0.7192 & 0.1089 &  &  &  &  \\
2000                                   & 4726  & 1088 & 4292  & \multicolumn{1}{l|}{898}   & 1.4792 & 0.3413 & 1.1107 & 0.2157 &  &  &  &  \\
3000                                   & 3362  & 897  & 3234  & \multicolumn{1}{l|}{771}   & 1.7550 & 0.5172 & 1.5711 & 0.3575 &  &  &  &  \\
4000                                   & 2663  & 859  & 2517  & \multicolumn{1}{l|}{668}   & 1.9415 & 0.6396 & 1.6634 & 0.4807 &  &  &  &  \\
5000                                   & 2398  & 826  & 2282  & \multicolumn{1}{l|}{605}   & 2.4134 & 0.8160 & 2.1528 & 0.5801 &  &  &  &  \\
6000                                   & 2100  & 772  & 2018  & \multicolumn{1}{l|}{586}   & 2.6235 & 0.8912 & 2.0975 & 0.6486 &  &  &  &  \\
7000                                   & 1970  & 752  & 1829  & \multicolumn{1}{l|}{562}   & 2.6019 & 1.0720 & 2.5441 & 0.7822 &  &  &  &  \\
8000                                   & 1861  & 747  & 1703  & \multicolumn{1}{l|}{555}   & 3.1035 & 1.2421 & 2.4987 & 0.8390 &  &  &  &  \\
9000                                   & 1750  & 747  & 1612  & \multicolumn{1}{l|}{530}   & 3.0223 & 1.3055 & 2.6148 & 0.8730 &  &  &  &  \\ \cline{1-9}
                                       &       &      &       &                            &        &        &        &        &  &  &  &
\end{tabular}
\end{table}
\begin{figure}[H]
\label{figure2}
  \centering
  \subfigure[]{
    \includegraphics[width=2.75in]{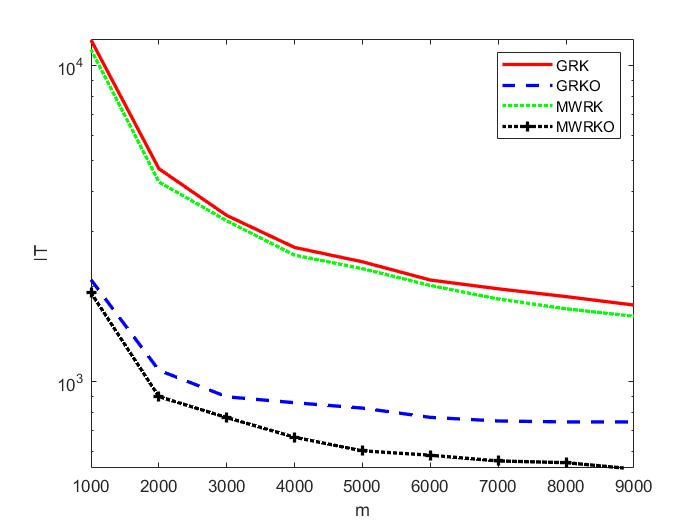}
  }
  \subfigure[]{
    \includegraphics[width=2.75in]{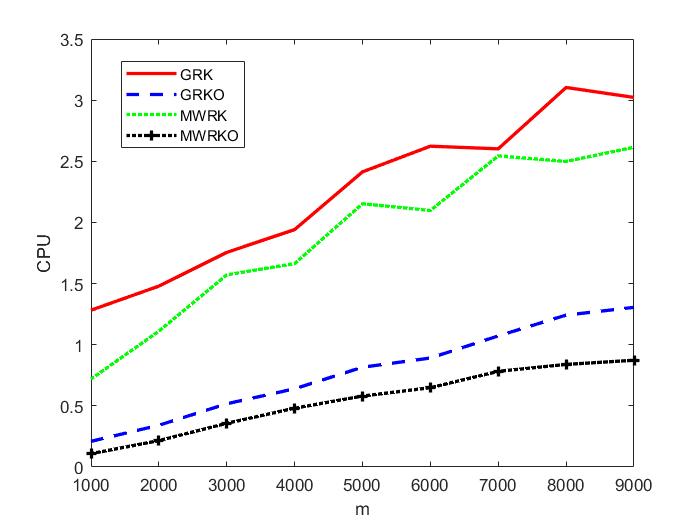}
  }
  \caption{Matrix $A\in R^{m\times 500}$ is generated by the $rand$ function in the interval $[0,1]$.
   (a): IT of the GRK, GRKO, MWRK, MWRKO methods changes with $m$.  (b): CPU of the GRK, GRKO, MWRK, MWRKO methods changes with $m$. }
\end{figure}
\begin{table}[H]
\label{table2}
\centering
\caption{IT and CPU of GRK, GRKO, MWRK and MWRKO for $m \times n$ matrices $A$ with $n = 2000$ and different $m$ when the consistent linear system is underdetermined}
\renewcommand\arraystretch{0.9}
\begin{tabular}{lllll|llll}
\hline
\multicolumn{1}{c}{\multirow{2}{*}{m}} & \multicolumn{4}{c|}{IT}      & \multicolumn{4}{c}{CPU}             \\ \cline{2-9}
\multicolumn{1}{c}{}                   & GRK   & GRKO & MWRK  & MWRKO & GRK     & GRKO   & MWRK    & MWRKO  \\ \hline
100                                    & 802   & 286  & 848   & 272   & 0.0496  & 0.0223 & 0.0258  & 0.0165 \\
200                                    & 1968  & 523  & 1948  & 481   & 0.1648  & 0.0496 & 0.0831  & 0.0276 \\
300                                    & 3104  & 759  & 3148  & 709   & 0.3982  & 0.1090 & 0.2404  & 0.0664 \\
400                                    & 4586  & 1002 & 4612  & 930   & 1.0539  & 0.2594 & 0.8433  & 0.1920 \\
500                                    & 6233  & 1250 & 6336  & 1215  & 1.9528  & 0.4409 & 1.6836  & 0.3576 \\
600                                    & 8671  & 1576 & 8882  & 1497  & 3.6363  & 0.7493 & 3.1625  & 0.5957 \\
700                                    & 11895 & 2063 & 11575 & 1879  & 5.8642  & 1.1078 & 5.0029  & 0.9087 \\
800                                    & 14758 & 2451 & 14888 & 2394  & 8.4280  & 1.5350 & 7.7007  & 1.6405 \\
900                                    & 18223 & 3250 & 18608 & 2945  & 12.0469 & 2.2750 & 10.9511 & 1.8884 \\ \hline
\end{tabular}
\end{table}
\begin{figure}[H]
\label{figure3}
  \centering
  \subfigure[]{
    \includegraphics[width=2.75in]{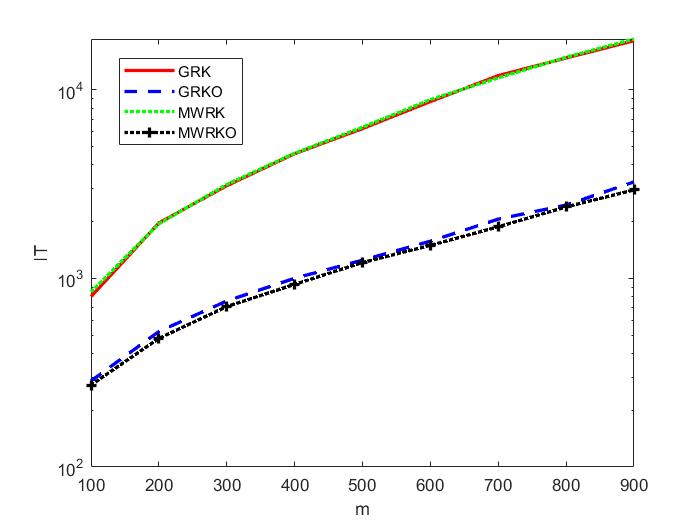}
  }
  \subfigure[]{
    \includegraphics[width=2.75in]{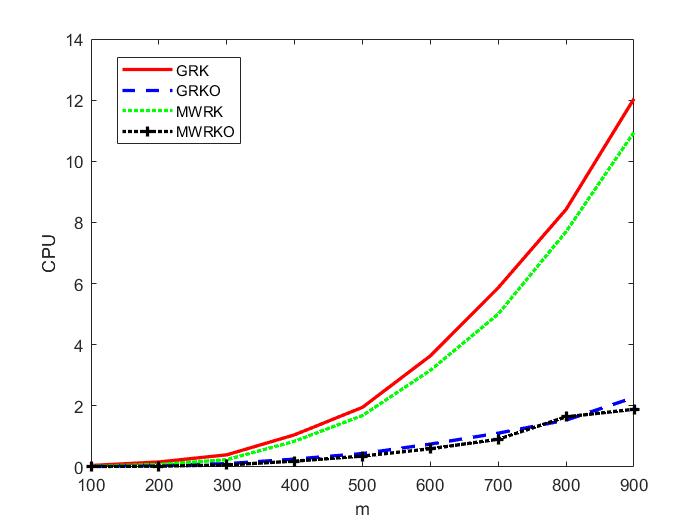}
  }
  \caption{Matrix $A\in R^{m\times 2000}$ is generated by the $rand$ function in the interval $[0,1]$.
   (a): IT of the GRK, GRKO, MWRK, MWRKO methods changes with $m$.  (b): CPU of the GRK, GRKO, MWRK, MWRKO methods changes with $m$. }
\end{figure}
\subsection{Experiments for Random Matrix Collection in $[c,1]$}
In this subsection, the entries of our coefficient matrix are randomly generated in the interval $[c,1]$. This set of experiments was also done in \cite{NW13} and \cite{WWT21}, and pointed out that when the value of $c$ is close to $1$, the rows of matrix $A$ is closer to linear correlation. Theorem 3.1 and theorem 3.2 have shown the effectiveness of the GRKO method and the MWRKO method in this case. In order to verify this phenomenon, we construct several $1000\times500$ and $500\times 1000$ matrices $A$, which entries is independent identically distributed uniform random variables on some interval $[c,1]$. Note that there is nothing special about this interval, and other intervals yield the same results when the interval length remains the same. In the experiment of this subsection, we take $\omega=0.5\times10^{-8}$.

From Table 3 and Figure 4 , it can be seen that when the linear system is overdetermined, with $c$ getting closer to $1$, the GRK method and the MWRK method have a significant increase in the number of iterations and running time. When $c$ increases to $0.7$, the GRK method and the MWRK method exceeds the maximum number of iterations.  But the IT and CPU of the GRKO method and the MWRKO method have decreasing trends. From Table 4 and Figure 5, we can get that the numerical experiment of the coefficient matrix $A$ in the underdetermined case has similar laws to the numerical experiment in the overdetermined case.

In this group of experiments, it can be observed that when the rows of the matrix are close to linear correlation, the GRKO method and the MWRKO method can find the least Euclidean norm solution more quickly than the GRK method and the MWRK methd.
\begin{table}[H]
\label{table3}
\centering
\caption{IT and CPU of GRK, GRKO, MWRK and MWRKO for matrices $A\in R^{1000\times500}$ generated by the $rand$ function in the interval $[c,1]$}
\renewcommand\arraystretch{0.9}
\begin{tabular}{lllll|llll}
\hline
\multicolumn{1}{c}{\multirow{2}{*}{c}} & \multicolumn{4}{c|}{IT}      & \multicolumn{4}{c}{CPU}           \\ \cline{2-9}
\multicolumn{1}{c}{}                   & GRK   & GRKO & MWRK  & MWRKO & GRK    & GRKO   & MWRK   & MWRKO  \\ \hline
0.1                                    & 14757 & 2036 & 14594 & 1830  & 1.5811 & 0.2180 & 0.9419 & 0.0969 \\
0.2                                    & 21103 & 1840 & 20717 & 1714  & 2.1684 & 0.2287 & 1.1828 & 0.1003 \\
0.3                                    & 27375 & 1708 & 26986 & 1569  & 3.5926 & 0.1789 & 1.5865 & 0.1195 \\
0.4                                    & 36293 & 1708 & 35595 & 1394  & 3.6751 & 0.1802 & 2.0682 & 0.0885 \\
0.5                                    & 53485 & 1428 & 52853 & 1310  & 5.3642 & 0.1486 & 3.0024 & 0.0847 \\
0.6                                    & 84204 & 1353 & 81647 & 1185  & 9.0879 & 0.1388 & 4.5468 & 0.0767 \\
0.7                                    & -     & 1227 & -     & 1036  & -      & 0.1298 & -      & 0.0564 \\
0.8                                    & -     & 1080 & -     & 926   & -      & 0.1107 & -      & 0.0580 \\
0.9                                    & -     & 715  & -     & 583   & -      & 0.0707 & -      & 0.0324 \\ \hline
\end{tabular}
\end{table}

\begin{figure}[H]
\label{figure4}
  \centering
  \subfigure[]{
    \includegraphics[width=2.75in]{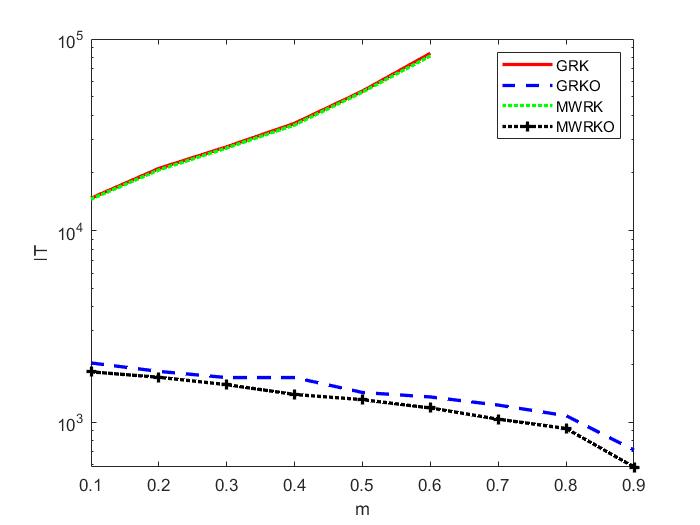}
  }
  \subfigure[]{
    \includegraphics[width=2.75in]{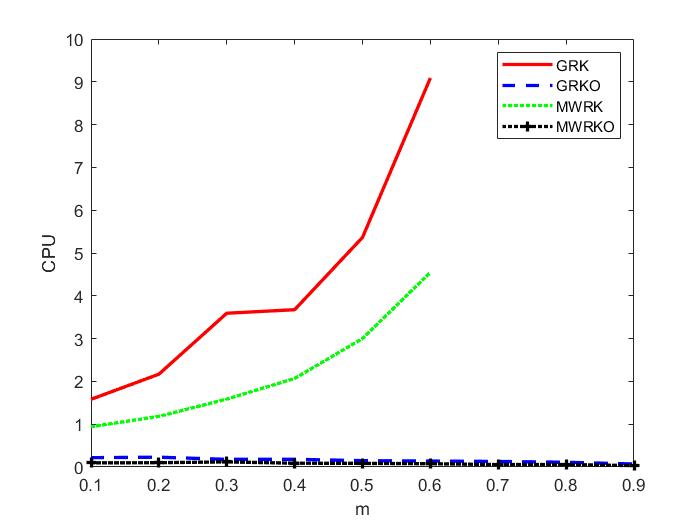}
  }
 \caption{Matrix $A\in R^{1000\times 500}$ is generated by the $rand$ function in the interval $[c,1]$.
   (a): IT of the GRK, GRKO, MWRK, MWRKO methods changes with $c$.  (b): CPU of the GRK, GRKO, MWRK, MWRKO methods changes with $c$. }
\end{figure}
\begin{table}[H]
\label{table4}
\centering
\caption{IT and CPU of GRK, GRKO, MWRK and MWRKO for matrices $A\in R^{500\times1000}$ generated by the $rand$ function in the interval $[c,1]$ }
\renewcommand\arraystretch{0.9}
\begin{tabular}{lllll|llll}
\hline
\multicolumn{1}{c}{\multirow{2}{*}{c}} & \multicolumn{4}{c|}{IT}      & \multicolumn{4}{c}{CPU}           \\ \cline{2-9}
\multicolumn{1}{c}{}                   & GRK   & GRKO & MWRK  & MWRKO & GRK    & GRKO   & MWRK   & MWRKO  \\ \hline
0.1                                    & 16828 & 1968 & 16913 & 1795  & 1.7612 & 0.2103 & 0.9353 & 0.1083 \\
0.2                                    & 23518 & 2003 & 23234 & 1857  & 2.3037 & 0.2066 & 1.3119 & 0.1230 \\
0.3                                    & 30875 & 1661 & 31017 & 1688  & 2.9310 & 0.1635 & 1.7373 & 0.0997 \\
0.4                                    & 41242 & 1511 & 40986 & 1515  & 4.3004 & 0.1726 & 2.2899 & 0.1025 \\
0.5                                    & 60000 & 1399 & 59750 & 1349  & 5.4754 & 0.1252 & 2.8920 & 0.0727 \\
0.6                                    & 97045 & 1270 & 95969 & 1264  & 8.5229 & 0.1173 & 4.8380 & 0.0688 \\
0.7                                    & -     & 1082 & -     & 1022  & -      & 0.1168 & -      & 0.0646 \\
0.8                                    & -     & 858  & -     & 863   & -      & 0.0960 & -      & 0.0585 \\
0.9                                    & -     & 549  & -     & 598   & -      & 0.0582 & -      & 0.0353 \\ \hline
\end{tabular}
\end{table}
\begin{figure}[H]
\label{figure5}
  \centering
  \subfigure[]{
    \includegraphics[width=2.75in]{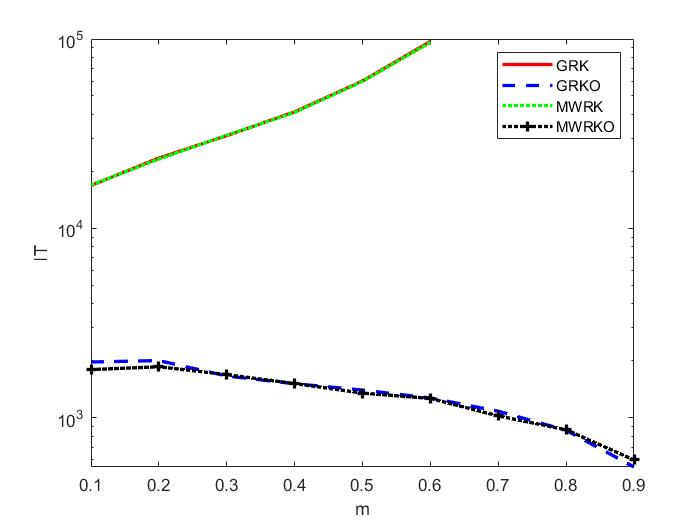}
  }
  \subfigure[]{
    \includegraphics[width=2.75in]{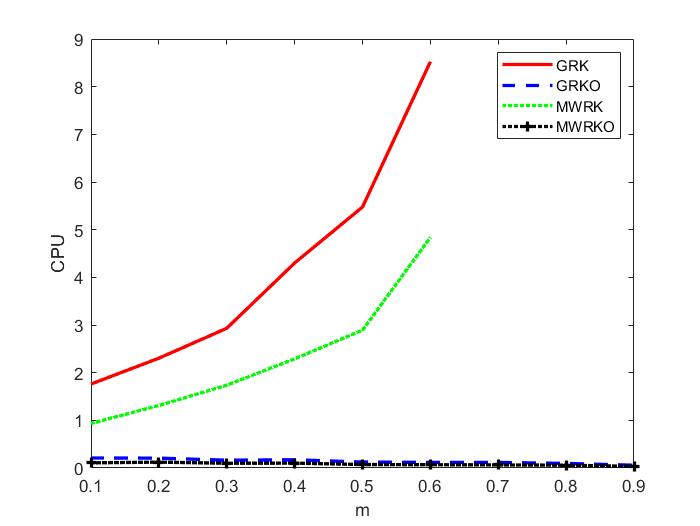}
  }
  \caption{Matrix $A\in R^{500\times 1000}$ is generated by the $rand$ function in the interval $[c,1]$.
   (a): IT of the GRK, GRKO, MWRK, MWRKO methods changes with $c$.  (b): CPU of the GRK, GRKO, MWRK, MWRKO methods changes with $c$. }
\end{figure}
\subsection{Experiments for Sparse Matrix}
In this subsection, we will give three examples to illustrate the effectiveness of our new methods applied to sparse matrix. The coefficient matrices $A$ of these three examples are the practical problems from \cite{TH11} and the two test problems from \cite{HP18}. We uniformly take $\omega=0.5\times10^{-5}$ in these three numerical examples.

\begin{example}\label{EX4.1}\upshape We solve the problem (\ref{Ax=b}) with the coefficient matrix $A\in R^{m\times n}$ chosen form the University of Florida sparse matrix collection \cite{TH11}. the matrices are $divorce$, $photogrammetry$, $Ragusa18$, $Trec8$, $Stranke94$, and $well1033$. In Table 5, we list some properties of these matrices, where density is defined as follows:
$$\text{density}=\frac{\text{number of nonzeros of m-by-n matrix}}{mn}.$$
\end{example}

In order to solve Example 4.1, we list the IT, CPU and historical convergence of the GRK, GRKO, MWRK, and MWRKO methods in Figure 6 and Table 6, respectively. It can be seen that MWRKO's IT and CPU are the least. Although the GRKO method is not faster than the MWRK method for most of the experiments in Table 6, it is always faster than the GRK method.
\begin{table}[H]
\label{table5}
\centering
\caption{The properties of different sparse matrices.}
\renewcommand\arraystretch{0.9}
\begin{tabular}{ccccccccc}
\cline{1-8}
\multicolumn{2}{c}{A}       & divorce & photogrammetry & Ragusa18 & Trec8   & Stranke94 & well1033 &  \\ \cline{1-8}
\multicolumn{2}{c}{m$\times$n}     & 50$\times$9    & 1388$\times$390       & 23$\times$23    & 23$\times$84   & 10$\times$10     & 1033$\times$320 &  \\
\multicolumn{2}{c}{rank}    & 9       & 390            & 15       & 23      & 10        & 320      &  \\
\multicolumn{2}{c}{cond$(A)$}       & 19.3908 & 4.35e+8    & 3.48e+35 & 26.8949 & 51.7330   & 166.1333 &  \\
\multicolumn{2}{c}{density} & 50.00\% & 2.18\%         & 12.10\%  & 28.42\% & 90.00\%   & 1.43\%   &  \\ \cline{1-8}
              &             &         &                &          &         &           &          &  \\
              &             &         &                &          &         &           &          &  \\
              &             &         &                &          &         &           &          &  \\
              &             &         &                &          &         &           &          &  \\
              &             &         &                &          &         &           &          &  \\
              &             &         &                &          &         &           &          &
\end{tabular}
\end{table}

\begin{figure}[H]
\label{figure6}
  \centering
  \subfigure[]{
    \includegraphics[width=1.5in]{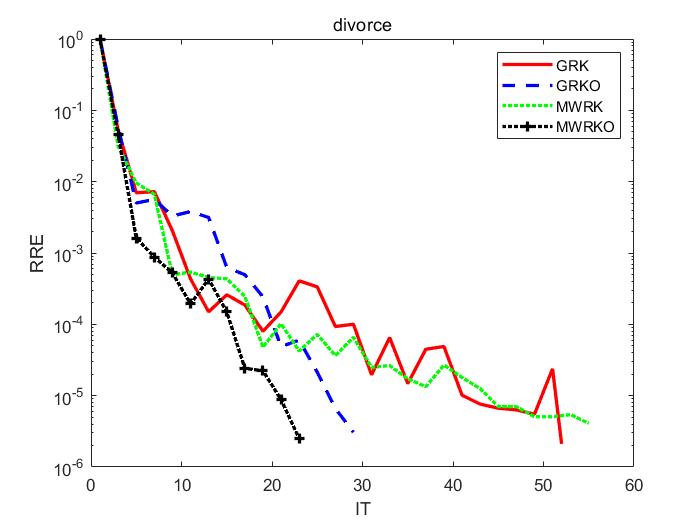}
  }
  \subfigure[]{
    \includegraphics[width=1.5in]{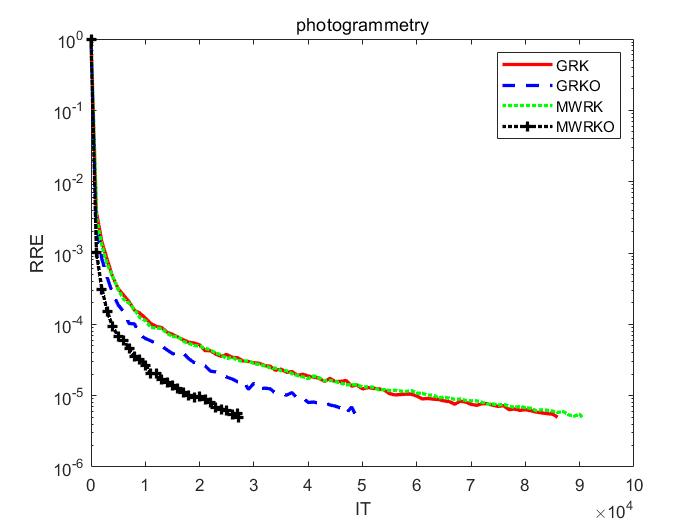}
  }
  \subfigure[]{
    \includegraphics[width=1.5in]{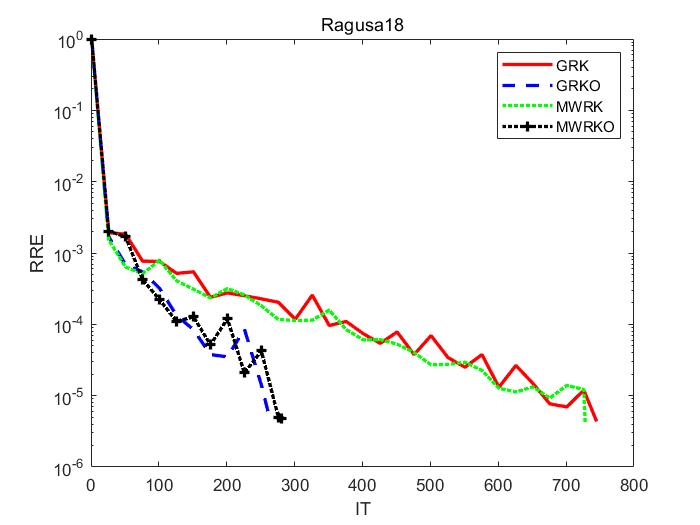}
  }
  \subfigure[]{
    \includegraphics[width=1.5in]{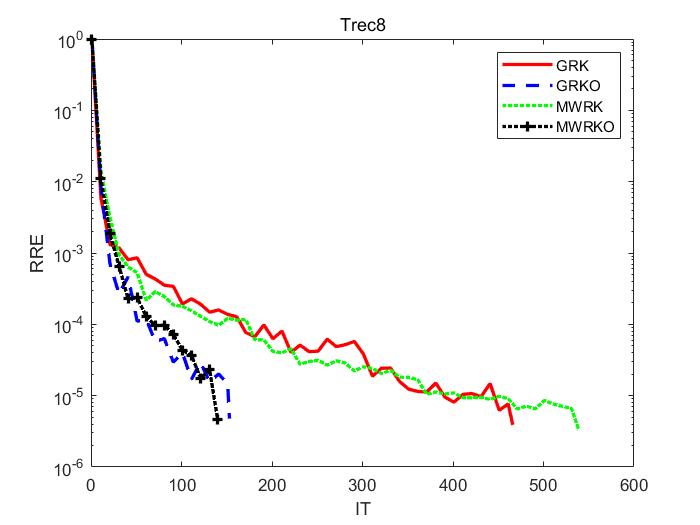}
  }
  \subfigure[]{
    \includegraphics[width=1.5in]{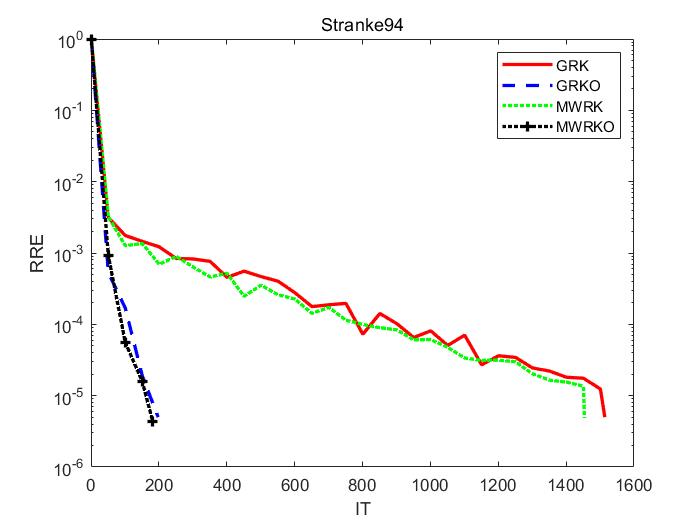}
  }
  \subfigure[]{
    \includegraphics[width=1.5in]{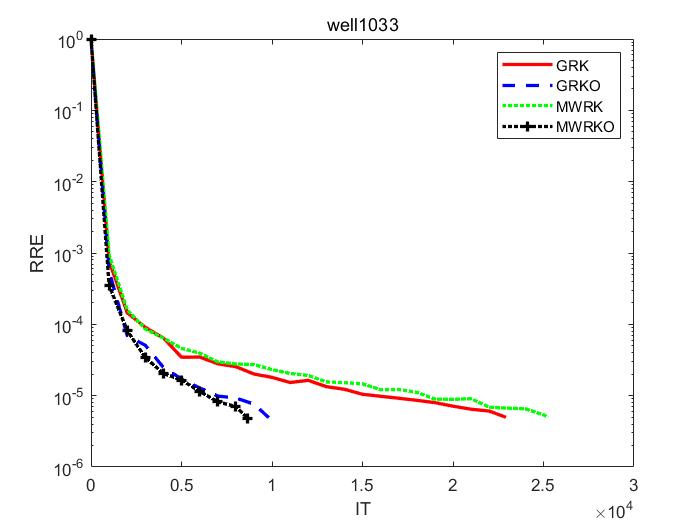}
  }
  \caption{ (a): Convergence history of methods for sparse matrices 'divorce'.  (b): Convergence history of methods for sparse matrices 'photogrammetry'.
   (c): Convergence history of methods for sparse matrices 'Ragusa18'.  (d): Convergence history of methods for sparse matrices 'Trec8'.
   (e): Convergence history of methods for sparse matrices 'Stranke94'.  (f): Convergence history of methods for sparse matrices 'well1033'.  }
\end{figure}
\begin{table}[H]
\label{table6}
\centering
\caption{IT and CPU of GRK, GRKO, MWRK and MWRKO for different sparse matrices }
\renewcommand\arraystretch{0.9}
\begin{tabular}{lllllllll}
\hline
\multicolumn{1}{c}{\multirow{2}{*}{A}} & \multicolumn{4}{c|}{IT}                            & \multicolumn{4}{c}{CPU}           \\ \cline{2-9}
\multicolumn{1}{c}{}                   & GRK   & GRKO  & MWRK  & \multicolumn{1}{l|}{MWRKO} & GRK    & GRKO   & MWRK   & MWRKO  \\ \hline
divorce                                & 51    & 28    & 54    & \multicolumn{1}{l|}{22}    & 0.0053 & 0.0037 & 0.0017 & 0.0013 \\
photogrammetry                         & 85938 & 48933 & 90480 & \multicolumn{1}{l|}{27084} & 9.9917 & 8.0424 & 3.9809 & 2.5026 \\
Ragusa18                               & 744   & 262   & 727   & \multicolumn{1}{l|}{280}   & 0.0577 & 0.0270 & 0.0121 & 0.0098 \\
Trec8                                  & 465   & 152   & 538   & \multicolumn{1}{l|}{139}   & 0.0382 & 0.0168 & 0.0111 & 0.0062 \\
Stranke94                              & 1513  & 197   & 1453  & \multicolumn{1}{l|}{181}   & 0.1291 & 0.0187 & 0.0208 & 0.0082 \\
well1033                               & 22924 & 9825  & 25250 & \multicolumn{1}{l|}{8655}  & 2.4278 & 1.5112 & 0.8491 & 0.5827 \\ \hline
                                       &       &       &       &                            &        &        &        &        \\
                                       &       &       &       &                            &        &        &        &        \\
                                       &       &       &       &                            &        &        &        &
\end{tabular}
\end{table}
\begin{example}\label{EX4.2}\upshape We consider $fancurvedtomo(N,\theta,P)$ test problem from the MATLAB package AIR Tools \cite{HP18}, which generates saprse matrix $A$, an exact solution $x^*$ and $b=Ax^*$. We set $N=60$, $\theta=0:0.5:179.5^{\circ}$, $P=50$, then resulting matrix is of size $32400\times3600$. We test  RRE every $10$ iterations and run these four methods until RRE$<\omega$ is satisfied, where $\omega=0.5\times10^{-5}$.
\end{example}

We first remove the rows of $A$ where the entries are all 0, and perform row unitization processing on $A$ and $b$. We emphasized that this will not cause a change in $x^*$. In Figure 7, we give $60\times60$ images of the exact phantom and the approximate solutions obatined by the GRK, GRKO, MWRK, MWRKO methods. In Figure 7, these four methods can basically restore the original image, but in the subgraph (f) of Figure 7, we can see that the MWRKO methods needs the least iterative steps, and the GRKO method has less iterative steps than GRK method. It can be observed from Table 7 that the MWRKO method is the best in terms of IT and CPU.
\begin{figure}[H]
\label{figure7}
  \centering
  \subfigure[]{
    \includegraphics[width=1.5in]{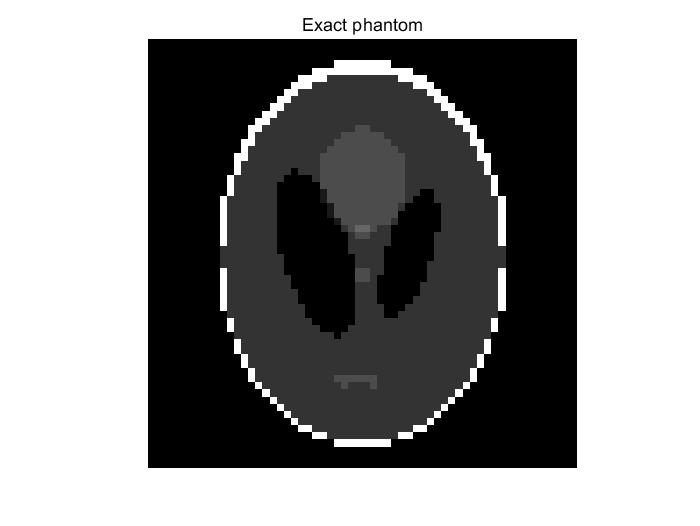}
  }
  \subfigure[]{
    \includegraphics[width=1.5in]{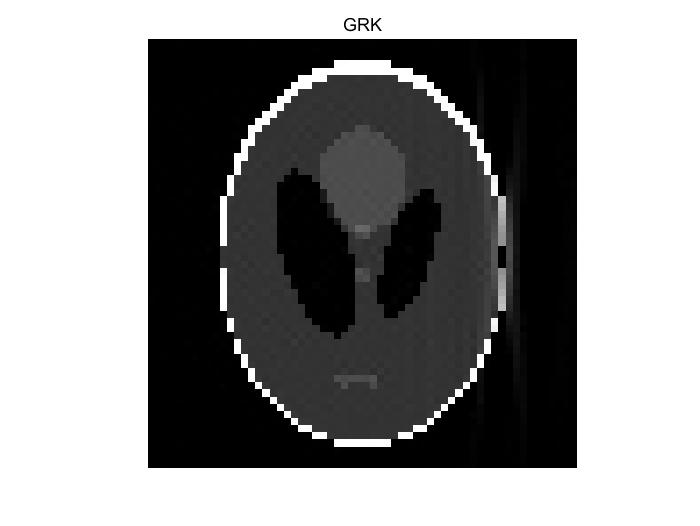}
  }
  \subfigure[]{
    \includegraphics[width=1.5in]{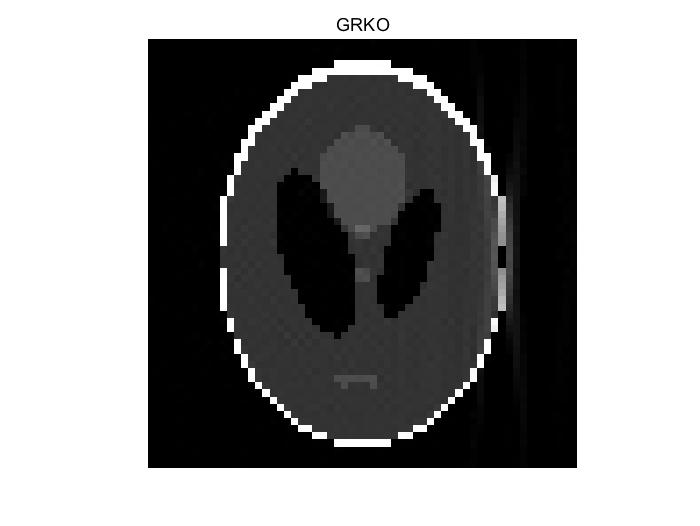}
  }
  \subfigure[]{
    \includegraphics[width=1.5in]{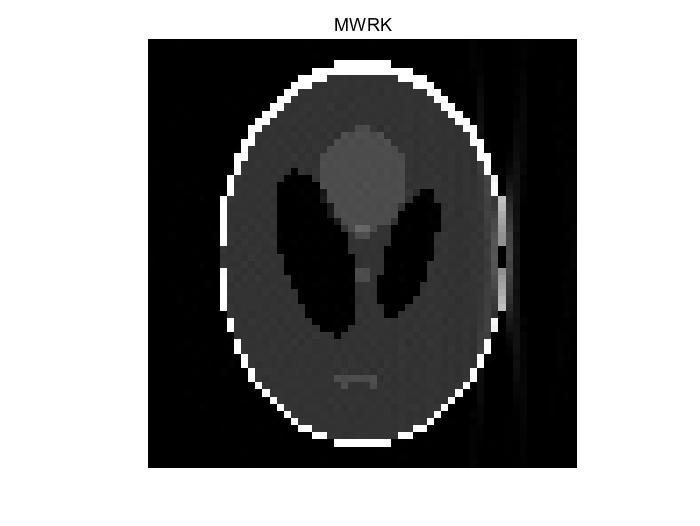}
  }
  \subfigure[]{
    \includegraphics[width=1.5in]{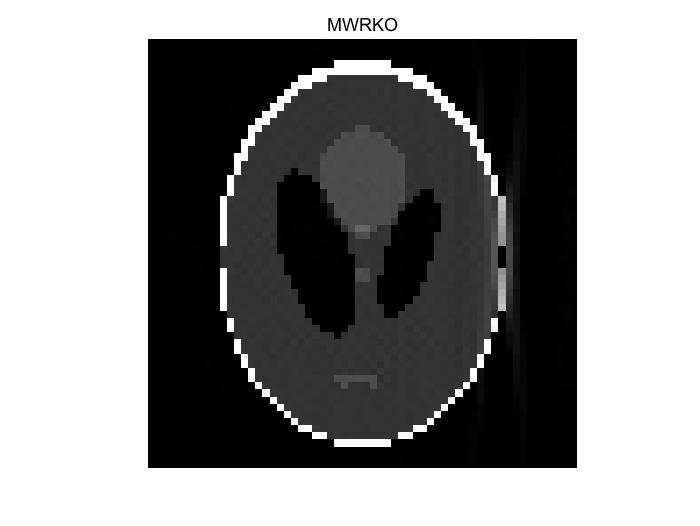}
  }
  \subfigure[]{
    \includegraphics[width=1.5in]{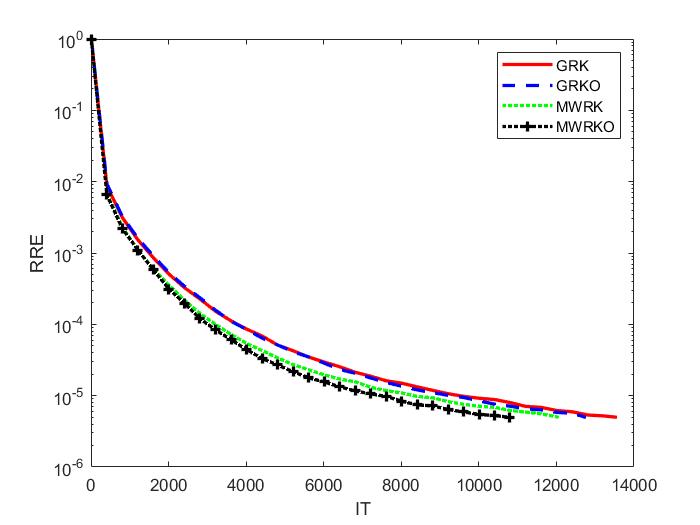}
  }
  \caption{Performance of GRK, GRKO, MERK, MWRKO methods for $fancurvedtomo$ test problem.
   (a):Exact phantom.  (b): GRK.
   (c):GRKO.           (d): MWRK.
   (e):MERKO.          (f): Convergence history of GRK, GRKO, MWRK, MWRKO methods.}
\end{figure}

\begin{table}[H]
\label{table7}
\centering
\caption{IT and CPU of GRK, GRKO, MWRK and MWRKO for $fancurvedtomo$ test problem }
\renewcommand\arraystretch{0.9}
\begin{tabular}{lllllllll}
\cline{1-3}
method & IT    & CPU      &  &  &  &  &  &  \\ \cline{1-3}
GRK    & 13550 &  581.17&  &  &  &  &  &  \\
GRKO   & 12750 & 538.82&  &  &  &  &  &  \\
MWRK   & 12050 & 504.83 &  &  &  &  &  &  \\
MWRKO  & 10790 & 452.86 &  &  &  &  &  &  \\ \cline{1-3}
       &       &          &  &  &  &  &  &  \\
       &       &          &  &  &  &  &  &  \\
       &       &          &  &  &  &  &  &  \\
       &       &          &  &  &  &  &  &  \\
       &       &          &  &  &  &  &  &  \\
       &       &          &  &  &  &  &  &
\end{tabular}
\end{table}
\begin{example}\label{EX4.3}\upshape We use an example from 2D seismic travel-time tomography reconstruction, implemented in the function $seismictomo(N,s,p)$ in the MATLAB package AIR Tools \cite{HP18}, which generates sparse matrix $A$, an exact solution $x^*$ and $b=Ax^*$. We set $N=12$, $s=24$, $p=35$, then resulting matrix is of size $840\times144$. We run these four methods until RRE$<\omega$ is satisfied, where $\omega=0.5\times10^{-5}$.
\end{example}

   We first remove the rows of $A$ where the entries are all 0, and perform row unitization processing on $A$ and $b$. In Figure 8, we give $12\times12$ images of the exact phantom and the approximate solutions obatined by the GRK, GRKO, MWRK, MWRKO methods. From the subgraph (f) of Figure 8 and Table 8, we can see that the MRKO method, the GRKO method, and the MWRK method perform similarly in the number of iteration steps, and are twice as small as the number of iteration steps of the GRK method. It can be observed from Table 8 that MWRKO method is the best in terms of IT and CPU.
\begin{figure}[H]
\label{figure8}
  \centering
  \subfigure[]{
    \includegraphics[width=1.5in]{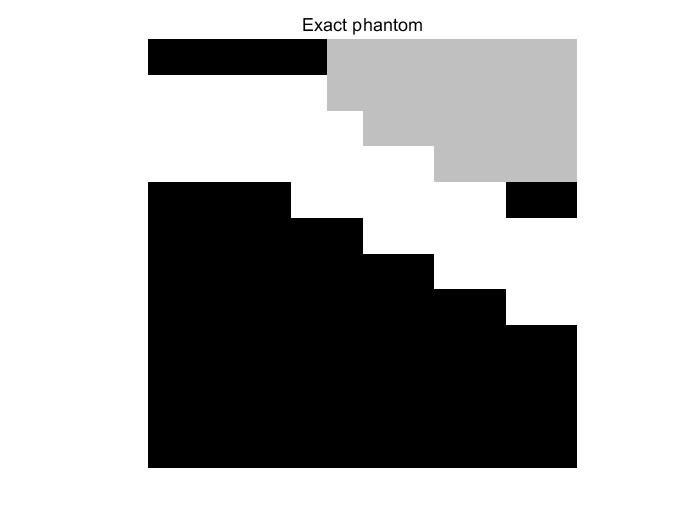}
  }
  \subfigure[]{
    \includegraphics[width=1.5in]{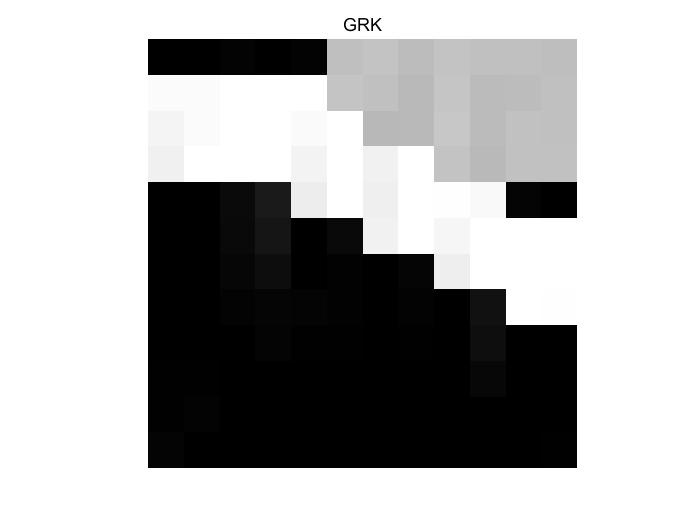}
  }
  \subfigure[]{
    \includegraphics[width=1.5in]{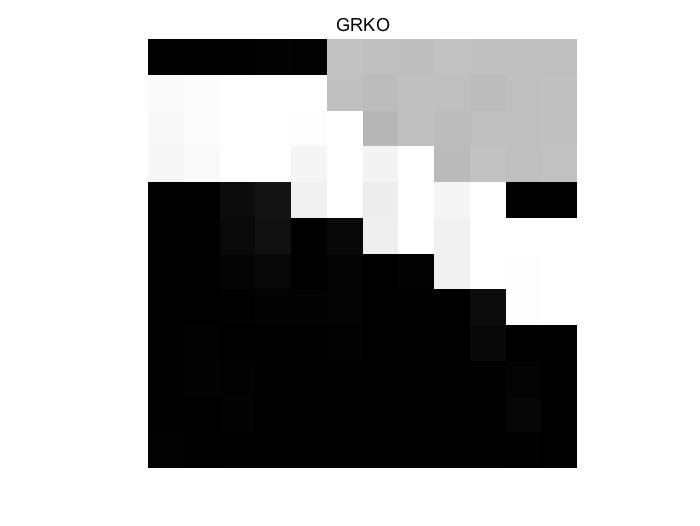}
  }
  \subfigure[]{
    \includegraphics[width=1.5in]{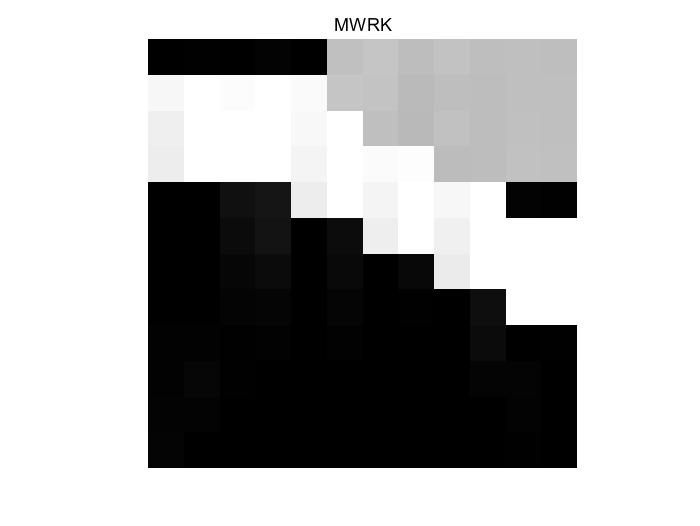}
  }
  \subfigure[]{
    \includegraphics[width=1.5in]{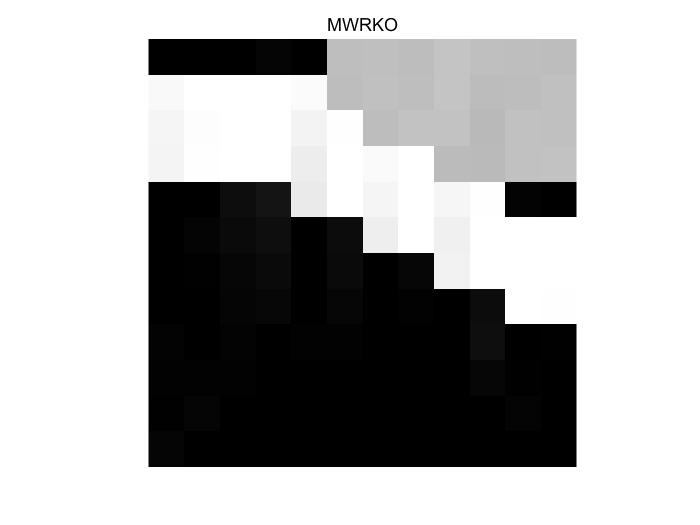}
  }
  \subfigure[]{
    \includegraphics[width=1.5in]{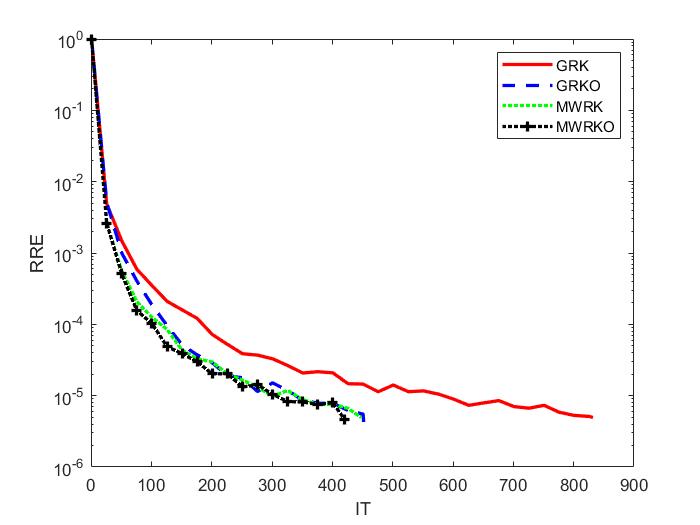}
  }
 \caption{Performance of GRK, GRKO, MERK, MWRKO methods for $seismictomo$ test problem.
   (a):Exact phantom.  (b): GRK.
   (c):GRKO.           (d): MWRK.
   (e):MERKO.          (f): Convergence history of GRK, GRKO, MERK, MWRKO methods.}
\end{figure}
\begin{table}[H]
\label{table8}
\centering
\caption{IT and CPU of GRK, GRKO, MWRK and MWRKO for $seismictomo$ test problem }
\renewcommand\arraystretch{0.9}
\begin{tabular}{lllllllll}
\cline{1-3}
method & IT    & CPU      &  &  &  &  &  &  \\ \cline{1-3}
GRK    & 831 & 0.0443 &  &  &  &  &  &  \\
GRKO   & 452 & 0.0273 &  &  &  &  &  &  \\
MWRK   & 447& 0.0125 &  &  &  &  &  &  \\
MWRKO  & 420 & 0.0108 &  &  &  &  &  &  \\ \cline{1-3}
       &       &          &  &  &  &  &  &  \\
       &       &          &  &  &  &  &  &  \\
       &       &          &  &  &  &  &  &  \\
       &       &          &  &  &  &  &  &  \\
       &       &          &  &  &  &  &  &  \\
       &       &          &  &  &  &  &  &
\end{tabular}
\end{table}
\section{Conclusion}
Combined with the representative randomized and non-randomized row index selection strategies, two Kaczamrz-type methods with oblique projection for solving large-scale consistent linear systems are proposed, namely the GRKO method and the MWRKO method. The exponential convergence of the GRKO method and the MWRKO method are deduced. Theoretical and experimental results show that the convergence rates of the GRKO method and the MWRKO method are better than GRK method and the MWRK method respectively. Numerical experiments show the effectiveness of these two methods, especially when the rows of the coefficient matrix $A$ are close to linear correlation.
\section*{Acknowledgments}
\hspace{1.5em}This work was supported by the Fundamental Research Funds for the Central Universities [grant number 19CX05003A-20], the National Key Research and Development Program of China [grant number 2019YFC1408400], and the Science and Technology Support Plan for Youth Innovation of University in Shandong Province [No.YCX2021151].

\end{document}